\documentclass[a4paper, 10pt]{amsart}

\usepackage{amsthm, amsmath, amssymb}
\usepackage{xcolor}
\usepackage{enumitem}

\usepackage[noadjust]{cite}
\usepackage[colorlinks, citecolor=red, urlcolor=blue, bookmarks=false, hypertexnames=true]{hyperref}

\theoremstyle{plain}
\newtheorem{theorem}{Theorem}[section]
\newtheorem*{theorem*}{Theorem \ref{thm:appl}}
\newtheorem{proposition}[theorem]{Proposition}

\newtheorem{corollary}[theorem]{Corollary}
\newtheorem{lemma}[theorem]{Lemma}

\theoremstyle{definition}

\numberwithin{equation}{section}

\newcommand{\inner}[2]{\langle #1,#2 \rangle}

\DeclareMathOperator{\trace}{trace}
\DeclareMathOperator{\grad}{grad}

\DeclareMathOperator{\Div}{div}

\DeclareMathOperator{\Hess}{Hess}

\allowdisplaybreaks

\title[Compact biconservative hypersurfaces]{ compact spacelike biconservative hypersurfaces in de Sitter space}
\author{Aykut Kayhan}
\address{Permanent address: Mathematics and Science Education, Maltepe University, Istanbul, 34480, Turkey}
\address{Current address: Faculty of Mathematics, Al. I. Cuza University of Iasi, Blvd. Carol I, no. 11, 700506 Iasi, Romania} \email{aykutkayhan@maltepe.edu.tr}
	
\subjclass[2020]{Primary 53C42. Secondary 53C40.}
\keywords{biconservative hypersurfaces, compact, spacelike, de Sitter}

\thanks{The  author was supported by the Scientific and Technological Research Council of T\"{u}rkiye (T\"{U}B\.{I}TAK), under the 2219-International Post-Doctoral Research Fellowship Programme (grant no: 1059B192300478). The opinions and views expressed herein are those of the authors and do not reflect those of T\"{U}B\.{I}TAK}

\begin{document}
	\maketitle
	
	\begin{abstract}
In this paper, we investigate the geometry of compact spacelike biconservative hypersurfaces with constant scalar curvature in de Sitter space $\mathbb{S}_1^{m+1}(c)$, under some geometric constraints. Our results extend the understanding of rigidity properties of such hypersurfaces in pseudo-Riemannian settings.
	\end{abstract}
	
	\section{Introduction}
	 Pseudo Riemannian manifold $N_1^{m+1}(c)$ with index 1 and of constant curvature $c$,  depending on whether $c>0,\ c=0,\ c<0$, is called de Sitter space, Minkowski space, or anti-de Sitter space, respectively, and is denoted by $\mathbb{S}_1^{m+1}(c),\ \mathbb{R}_1^{m+1}(c)$ or $\mathbb{H}_1^{m+1}(c)$. These
	 three spacetimes are known as the Lorentzian space forms.
	
Let $M^m$ be $m$-dimensional hypersurface in a $(m+1)$-dimensional Lorentzian space form $N_1^{m+1}(c)$ 	with mean curvature vector $H=f\eta$ and shape operator $A_\eta$ such that $\eta$ is unit normal vector field and $\inner{\eta}{\eta}=\varepsilon=\pm 1$. $M^m$ is said to be biconservative if and only if
\begin{equation}\label{Agradf}
A_\eta(\grad f)+\varepsilon\frac{mf}{2}\grad f=0.
\end{equation}

  Biconservative manifolds have attracted the attention of many geometers in both Riemannian and Pseudo-Riemannian setting in last three decade years (For a detailed summary, see \cite{chenrecentdevelopment}). Although there are numerous studies in the literature in the Riemannian context, there are not as many studies in the Pseudo-Riemannian context. In particular, if the hypersurface is spacelike (i.e., $\varepsilon=-1$) the number of studies conducted is even more limited, as we can summarize below:

Yu Fu classified the  timelike ($\varepsilon=1$) and spacelike  surfaces, known as non-degenerate surfaces, in 3-dimensional Lorentizan space forms in \cite{fu2013bi} and \cite{fu2015explicit}. He demonstrated that such surfaces are either CMC or rotational. Later, Yu Fu and Turgay provided examples of spacelike biconservative hypersurfaces while attempting to classify such hypersurfaces in Minkowski 4-space with a diagonalizable shape operator and two distinct principal curvatures \cite{fu2016complete}. For more, see \cite{chenrecentdevelopment}.

	In a more general sense,  As we mentioned before, a hypersurface $M^m$ of $N_1^{m+1}(c)$ is called spacelike if $\varepsilon=-1$, or equivalently, the metric induced on $M^m$ from that of ambient space is positive definite. Spacelike hypersurfaces are fundamental in general relativity, serving as initial data surfaces for the Cauchy problem in arbitrary spacetimes. They provide a natural framework for studying gravitational wave propagation and the dynamical evolution of spacetime curvature. Existence and uniqueness results for such hypersurfaces have been established under various geometric and analytic conditions on the ambient spacetime, as demonstrated in works like \cite{B}, \cite{CB}, \cite{CFM}, \cite{N} and \cite{S}. 
	
When the ambient space is de Sitter space, Zheng \cite{zheng1995}  considered $m-$dimensional compact spacelike  hypersurfaces in de Sitter space $\mathbb{S}_1^{m+1}(c)$ with constant scalar curvature and obtained the following theorem.
	\begin{theorem}\cite{zheng1995}\label{zhengtheo}
		Let $M^m$ be an $m$-dimensional compact spacelike hypersurface immersed in de Sitter space $\mathbb{S}_1^{m+1}(c)$ with constant scalar curvature $m(m-1)r$. If $M^m$ has non-negative sectional curvature  and satisfies $r<c$ then $M^m$ isometric to a sphere.
	\end{theorem}
	In this paper, we interested in the study of compact spacelike \textit{biconservative} hypersurfaces in de Sitter space with constant scalar curvature. To accomplish this, replacing the hyphotesis $r<c$ with  biconservativity, we obtained the following results deduced from the main theorem of this paper (see Theorem \ref{maintheo}).
	\begin{corollary}\label{firsttheo}
		Let $\varphi:M^m\to \mathbb{S}_1^{m+1}(c)$ be compact spacelike  hypersurface  in de Sitter space $\mathbb{S}_1^{m+1}(c)$ with constant scalar curvature $m(m-1)r$ and non-negative sectional curvature.  If $M^m$ is non-minimal biconservative then $M^m$ isometric to a sphere  $\mathbb{S}^m(c_1),\ 0<c_1<c$.
	\end{corollary}
	
		\begin{corollary}\label{licorollary}
		Let $\varphi:M^m\to \mathbb{S}_1^{m+1}(c)$ be compact spacelike  hypersurface  in de Sitter space $\mathbb{S}_1^{m+1}(c)$ with constant  scalar curvature.  If $M^m$ is non-minimal biconservative and $f^2\leq 4(m-1)c/m^2$, then $M^m$ isometric to a sphere  $\mathbb{S}^m(c_1),\ 0<c_1<c$.
	\end{corollary}
	
	In fact, Corollary \ref{licorollary} is a generalizaton of the theorem (see Theorem 4.1 in \cite{li1997}) given by Haizhong Li to biconservative hypersurfaces. Because 	Li used the condition $r<c$ in that theorem. By replacing this condition with biconservativity and having constant scalar curvature, we obtained a generalization of this theorem to biconservative hypersurfaces
	
Afterward, we aimed to relax the fact that scalar curvature is constant with the condition of $m(m-1)r=kf$( $k=\mbox{const }>0$). Such hypersurfaces in de Sitter space have been studied by Qing-ming Cheng \cite{cheng1990rkh}, and in the case where the hypersurface is complete, several results have been established. Therefore, this emphasizes the importance of investigating the compactness condition. In this manner, we obtained the following theorem.
	
\begin{theorem}\label{extra result}
	Let $\varphi:M^m\to\mathbb{S}_1^{m+1}(c)$ be a compact spacelike  hypersurface in de Sitter space with non-negative sectional curvature. If $m(m-1)r=kf$( $k=\mbox{const }>0$) and $M$ is biconservative then $M$ is isometric to a sphere $\mathbb{S}^{m}(c_0),\ 0<c_0<c$.
\end{theorem}	
	 	
	Finally, we gave a result for compact spacelike biconservative surface  in 3-dimensional de Sitter space .
	\begin{theorem}\label{surfacetheo}
		Let $M$ be a compact spacelike biconservative surface in 3-dimensional de Sitter space. Then $M$ is totally umbilical.
	\end{theorem}
	\section{Preliminaries}
	Consider the real vector space $\mathbb{R}^{m+2}$ equipped with the Lorentzian metric $\langle , \rangle$ given by
	$$\inner vw=-v_0w_0+\sum_{i=1}^{m+1}v_iw_i$$
	for any $v,w\in\mathbb{R}^{m+2}$. With this metric, $\mathbb{R}_1^{m+2}$ is referred to the $(m+2)$-dimensional Minkowski space. Then, the de Sitter space $\mathbb{S}_1^{m+1}(c)$ is defined by
	$$\mathbb{S}_1^{m+1}(c)=\{x\in \mathbb{R}_1^{m+2};\inner xx=\frac{1}{c}\}.$$
	In this way, the de Sitter space $\mathbb{S}_1^{m+1}(c)$ inherits from  $\langle , \rangle$ a metric which makes it Lorentzian manifold with constant sectional curvature $c$. Let $\hat{\nabla}, \bar{\nabla}$ denote the metric connections of $\mathbb{R}_1^{m+2}$ and $\mathbb{S}_1^{m+1}(c)$, respectively, we have
	\begin{equation}\label{connection relation}
		\hat{\nabla}_vw-\bar{\nabla}_vw=-c\inner{v}{w}x,
	\end{equation}
	where $v,w$ are vector fields tangent to $\mathbb{S}_1^{m+1}(c)$.

 Let $\nabla$ denote the Levi-Civita connection associated with the Riemannian metric induced on $M$ by the Lorentzian metric $\langle ,\rangle$. Then, the Weingarten endomorphism $A$ of $\varphi$ is expressed by
 \begin{eqnarray}\label{weingarten endo}
 \bar{\nabla}_v\eta=-Av,\ \ \bar{\nabla}_vw-\nabla_vw=-\inner{Av}{w}\eta,
 \end{eqnarray}
 where $v,w$ are tangent to $M^m$. Using \eqref{weingarten endo}, we deduce that the mean curvature function $f$ is given by
\begin{equation}\label{trace A}
\trace A=-mf.
\end{equation}

Let $R$ denote the curvature tensor field of $M^m$. Then we have
 \begin{equation}
 \label{gauss}	R(X,Y)Z=c\{ \inner{Y}{Z}X-\inner{X}{Z} Y\}-\{ \inner{AY}{Z}AX-\inner{AX}{Z} AY \},
 \end{equation}
 where $X,Y$ and $Z$ are vectors fields tangent to $M^m$.  The Codazzi equation is expressed by 
 $$(\nabla_XA)Y=(\nabla_YA)X.$$
 Taking a local orthonormal frame field $\{E_1,\cdots,E_m\}$ on $M$ which diagonalizes $A$, we obtain from \eqref{gauss} that
 \begin{equation}\label{normalizedscalar}
 	m(m-1)(c-R)=m^2f^2-\vert A\vert^2,
 \end{equation}
where $\vert A\vert^2=\sum\limits_{i,j=1}^m\vert\inner{AE_i}{E_j}\vert^2$ and $R$ is called the normalized scalar curvature given by the  scalar curvature divided by $m(m-1)$ .

Hessian and Laplacian of a $C^2$- function $\alpha$ defined on $M$ is given by
\begin{eqnarray*}
	\Hess\alpha(Y,X)&=&X(Y\alpha)-(\nabla_XY)\alpha,\\
-	\Delta\alpha&=&\trace\Hess\alpha.
\end{eqnarray*}

In order to investigate the properties of a compact spacelike biconservative hypersurfaces, we use the well-established Cheng-Yau formula is given by 
	\begin{equation}\label{Cheng-Yau formula}
-\frac 1 2\Delta\vert S\vert^2=\vert \nabla S \vert ^2+\langle S,\Hess \trace S\rangle+\frac 1 2 \sum_{i,j=1}^{m}R_{ijij}(\mu_i-\mu_j)^2,
	\end{equation}
	where $S$ is a symmetric $(1,1)$ tensor field on an arbitrary  Riemannian manifold $M$ satisfying $(\nabla_XS)Y=(\nabla_YS)X$ and $\mu_i$'s are the eigenvalues of $S$.
	
	Another well-known tool is the Cheng-Yau operator $\square$ associated to a symmetric $(1,1)$ tensor field $\phi$, (see \cite{cheng1977hypersurfaces}). $\square\alpha$ is expressed by
	\begin{equation}
		\square\alpha=\inner{\phi}{\Hess\alpha}
	\end{equation}
	for any $\alpha\in C^2(M)$. If $\phi$ is a divergence-free tensor defined on a compact manifold, then $\square$ is self adjoint, i.e.
	$$\int_M\alpha(\square\beta)=\int_M\beta(\square\alpha).$$
	From which, one can conclude that
	\begin{equation}
\int_M\square\alpha=0.
	\end{equation}
	Before proceeding we would like you to notice that the equation \eqref{Agradf} becomes
	\begin{equation}\label{Agradffor spacelike}
A(\grad f)=\frac{mf}{2}\grad f,
	\end{equation}
due to $M$ is spacelike. Another well-known property of the shape operator $A$ for spacelike hypersurfaces in Lorentzian space forms is that
\begin{equation}\label{div A}
\Div A=-m\grad f.
\end{equation}

 Our main tool is the effective use of the Cheng-Yau $\square$ operator. Notice that  $M$ is Riemannian manifold because of $M$ is spacelike then it has positive definite metric. So, we can use the Cheng-Yau technique as in \cite{rigidityourpaper}. To achieve this, we construct divergence-free tensor fields on $M$, which play a crucial role in classifying such hypersurfaces. In this direction, we present the following proposition.
	\begin{proposition}\label{f2A}
	Let	$\varphi:M^m\to N_1^{m+1}(c)$ be a spacelike biconservative hypersurface $M^m$ immersed into the Lorentzian space form $N_1^{m+1}(c)$. Let symmetric tensor field  $T_1$ is given by $T_1=f^2A$. Then
$$M\text{ is biconservative}\Leftrightarrow\Div T_1=0.$$
	\end{proposition}
	\begin{proof}
	By direct computations, we have
	\begin{align}
\Div T_1=2f\bigg\{ A(\grad f)-\frac{mf}{2}\grad f\bigg\}
	\end{align}
	If $M$ is biconservative then he proof is trivial for from \eqref{Agradffor spacelike}.  Now Suppose that $\Div T_1=0$ then
	\begin{equation}
		f=0\mbox{   or   }A(\grad f)=\frac{mf}{2}\grad f
	\end{equation}
	If $A(\grad f)=\frac{mf}{2}\grad f$ on $M$, then $\Div T_1=0.$ For the sake of contradiction assume that $A(\grad f)\neq\frac{mf}{2}\grad f$ at any point of $U$. Then $f=0$ on $U$, which means the $\grad f=0$ on $U$. So, $A(\grad f)=\frac{mf}{2}\grad f$ on $U$ which is a contradiction. Therefore, $\Div T_1=0$ implies that $M$ is biconservative.  
		\end{proof}
	Finally, Throughout this paper, we restrict our analysis to manifolds that are assumed to be connected. 
	\section{biconservative hypersurface with constant  scalar curvature }
In this section, we present the proofs of the theorems  given in the introduction. To do this first we need to the following theorem.
\begin{theorem}\label{maintheo}
	Let $\varphi:M^m\to S_1^{m+1}(c)$ be a compact spacelike  hypersurface in de sitter space with non-negative sectional curvature. If $M$ is non-minimal biconservative and has constant scalar curvature then $M$ is CMC.
\end{theorem} 
\begin{proof}
	We make use of Cheng-Yau square oparator $\square$  with divergence-free $(1,1)$ tensor $T_1$. So, we have
	\begin{eqnarray}
	\nonumber	\square mf&=&\langle T_1, \Hess mf\rangle\\
		\label{equfirst}	  &=&f^2\langle A,\Hess mf \rangle
	\end{eqnarray}
	Note that Cheng-Yau formula \eqref{Cheng-Yau formula} becomes
		\begin{equation}\label{Cheng-Yau formula New}
		-\frac 1 2\Delta\vert A\vert^2=\vert \nabla A \vert ^2-\langle A,\Hess mf\rangle+\frac 1 2 \sum_{i,j=1}^{m}R_{ijij}(\lambda_i-\lambda_j)^2,
	\end{equation}
	Putting  \eqref{Cheng-Yau formula New} into \eqref{equfirst}, we get
	\begin{equation}\label{mainequ}
		\square mf=f^2\bigg\{\frac 1 2\Delta \vert A\vert^2+\vert \nabla A\vert^2+\frac 1 2 \sum_{i,j=1}^{m}R_{ijij}(\lambda_i-\lambda_j)^2 \bigg\}
	\end{equation}
	Because of the normalized scalar curvature is constant, applying the Laplacian operator $\Delta$ to equation \eqref{normalizedscalar} and multiplying both side by $f^2$, we get
	\begin{equation}
		\frac{f^2}{2} \Delta\vert A\vert^2=\frac{m^2}{2} f^2\Delta f^2.
	\end{equation}
From which we have
\begin{equation}
\int_M	\frac{f^2}{2} \Delta\vert A\vert^2=\int_M\frac{m^2}{2}\langle \grad f^2,\grad f^2 \rangle=\int_M2m^2 f^2\vert \nabla f \vert^2.
\end{equation}
	So integrating \eqref{mainequ} and taking into account equality above, we obtain the following main formula.
	\begin{equation}
			\label{mainequ2}0=\int_M f^2\big\{2m^2\vert\nabla f\vert^2+\vert \nabla A\vert^2+\frac 1 2 \sum_{i,j=1}^{m}R_{ijij}(\lambda_i-\lambda_j)^2 \big\}.
	\end{equation}
	So, we have
	$$f^2\vert\nabla A\vert^2=0.$$
	We claim that $\grad f=0$. For the sake of contradiction, assume that $\grad f\neq 0$ at any point of $U$ in $M$. Eventually, by restricting $U$, we can assume that $f\neq 0$ at any point of $U$. Then $\nabla A=0$ on $U$ which implies $\grad f=0$ on $U$. This is a contradiction. So $M$ is CMC.
\end{proof}

\subsection{\textit{The proof of Corollary \ref{firsttheo}}}
	Because of $R$ and $f$ are constant, it is obvious that they are linearly related. Moreover, $f$ obtains its maximum on $M$.   From this and  compactness, Theorem \ref{firsttheo}  follows immediately from a result of 
	Cheng (see \cite{cheng1990rkh}, Theorem 1 ) .\hfill \qed
	
	Now, we would like to state another result of Theorem \ref{maintheo}  motivated by Haizhong Li( see \cite{li1997}, Theorem 4.1).

	\subsection{\textit{The proof of Corollary \ref{licorollary}}}
		We have from  p.~343 of Li~\cite{li1997} that
		\begin{align}\label{li1}
			\frac 12 \sum_{i,j}R_{ijij}(\lambda_i-\lambda_j)^2&\geq (\vert A\vert^2-mf^2)\bigg(mc-2mf^2+\vert A\vert^2-\frac{m(m-2)}{\sqrt{m(m-1)}}\vert f\vert \sqrt{\vert A\vert^2-mf^2}\bigg)
		\end{align}
		and 
		\begin{align}\label{li2}
			\bigg(mc-2mf^2+\vert A\vert^2&-\frac{m(m-2)}{\sqrt{m(m-1)}}\vert f\vert \sqrt{\vert A\vert^2-mf^2}\bigg)\\
			\nonumber	&=\bigg(\sqrt{\vert A\vert^2-mf^2}-\frac 12(m-2)\vert f\vert\sqrt{\frac{m}{m-1}}\bigg)^2+m\bigg(c-\frac{m^2}{4(m-1)}f^2\bigg)
		\end{align}
		As can be easily seen from \eqref{li1} and \eqref{li2}, the fact that $f^2\leq 4(m-1)c/m^2$ implies that $\sum\limits_{i,j}R_{ijij}(\lambda_i-\lambda_j)^2\geq 0$. The proof follows directly by performing \eqref{mainequ2}.\hfill \qed
		
	Now, we would like to give a result that we do not assume that the scalar curvature is constant. 
	\subsection{\textit{The proof of Theorem \ref{extra result}}}
Since the scalar curvature is the trace of the Ricci curvature, which itself arises as the trace of the sectional curvature, and assuming that the scalar curvature is non-negative, it follows that $f>0$. Actually; from Gauss equation,
\begin{equation}\label{gaussforr=kf}
	m(m-1)(c-kf)=m^2f^2-\vert A\vert^2
\end{equation}
If there exists a point p on $M$ so that $r=0$, we have $f=0$. \eqref{gaussforr=kf} implies
$$m(m-1)c+\vert A\vert^2=0,$$
which is impossible. So, $r>0$ and $f>0$.

Note that we have the equation $\eqref{mainequ}$ since $M$ is biconservative. Applying integration by part, \eqref{mainequ} becomes
\begin{equation}\label{mainequfor r=kf}
	0=\int_M\langle \grad f^2,\grad \vert A\vert^2\rangle +f^2\big\{  \vert \nabla A\vert^2+	\frac 12 \sum_{i,j}R_{ijij}(\lambda_i-\lambda_j)^2 \big\}
\end{equation}
We have from \eqref{gaussforr=kf} that
\begin{equation}\label{delta A^2}
	\grad\vert A\vert^2=m^2\grad f^2+k\grad f
\end{equation}
Substituting \eqref{delta A^2} into \eqref{mainequfor r=kf}, we get
\begin{equation*}
0=\int_M m^2\vert\grad f^2\vert^2+kf\vert\grad f\vert^2+f^2\big\{  \vert \nabla A\vert^2+	\frac 12 \sum_{i,j}R_{ijij}(\lambda_i-\lambda_j)^2 \big\}
\end{equation*}
This implies $\grad f=0$ at any point on $M$ since $f>0$ at any point on $M$. So $f$ obtains on maximum on $M$. Thus, because of the compactness $M$ is isometric to a sphere $\mathbb{S}^m(c_0)$,  $1<c_0<c$, from a result of Cheng (\cite{cheng1990rkh}, Theorem 1).\hfill \qed 
\section{Compact spacelike biconservative surfaces}
Before starting, we would like to refer to Yu Fu's works that he classified biconservative surfaces in 3-dimensional space forms in \cite{fu2013bi} and \cite{fu2015explicit}. Here, we shall classify the compact ones of them. In order to do this, making use of Chen inequality given in Riemmanian manifold, we give an useful inequalitiy between $\vert \nabla A\vert^2$ and $\vert \grad f\vert^2$. Afterward we apply for Cheng-Yau formula.
\begin{lemma}\label{cheninequfor spclike}
	Let $\varphi:M^m\to N_1^{m+1}(c)$ be compact spacelike  hypersurface. If $M^m$ is biconservative then
	\begin{equation}
		\vert \nabla A\vert^2\geq\frac{m^2(m+2)}{4(m-1)}\vert \grad f \vert^2
	\end{equation}
\end{lemma}
\begin{proof}
	Obviously, if $\grad f=0$ the inequality holds true automatically. Now, assume that $\grad f\neq 0$ at a point $p$ on $M$. Then $\grad f$ doesn't vanish throughout an open neighborhood of $p$. In this neighborhood, consider orthonormal frame  field $\{E_1=\frac{\grad f}{\vert\grad f\vert}, E_2,\cdots,E_m\}$. Then we have
	$$A(E_1)=\frac{mf}{2}E_1$$
	Now,
	\begin{align}
	\nonumber	\vert\nabla A\vert^2=&\sum_{i,j=1}^m\vert(\nabla A)(E_i,E_j)\vert^2=\sum_{i,j=1}^m\langle (\nabla A)(E_i,E_j),E_1\rangle^2\\
	\label{nabla A>}	&\geq\langle (\nabla A)(E_1,E_1),E_1\rangle+\frac{3}{m-1}\left(\sum_{i,j=1}^m\langle (\nabla A)(E_i,E_i),E_1\rangle\right)^2.
	\end{align}
	note that
	\begin{equation}\label{A11}
\vert \langle (\nabla A)(E_1,E_1),E_1\rangle\vert^2=\vert E_1\langle AE_1,E_1\rangle \vert^2=\vert E_1(\frac{mf}{2})\vert^2=\frac{m^2}{4}\vert \grad f\vert^2
	\end{equation}
	and
	\begin{eqnarray}\label{Aii1}
\nonumber\left(\sum_{i,j=2}^m \langle(\nabla A)(E_i,E_i),E_1\rangle\right)^2&=&\left(\sum_{i,j=2}^m E_1\langle A(E_i),E_i\rangle \right)^2\\
\nonumber&=&\vert E_1(mf-\lambda_1) \vert^2\\
&=&\frac{m^2}{4}\vert\grad f\vert^2
	\end{eqnarray}
	Substituting \eqref{A11} and \eqref{Aii1} into \eqref{nabla A>} we get
	\begin{align*}
		\vert\nabla A\vert^2\geq& \frac{m^2}{4}\vert\grad f\vert^2+\frac{3}{m-1}\frac{m^2}{4}\vert\grad f\vert^2\\
		&=\frac{m^2(m+2)}{4(m-1)}\vert\grad f\vert^2.
	\end{align*}
\end{proof}

Now, since $M$ is Riemannian manifold, we have
\begin{equation}\label{div A gradf}
	\Div A(\grad mf)=\langle \Div A,\grad mf\rangle+\langle A,\Hess mf\rangle.
\end{equation}
Because $M$ is biconservative then
\begin{eqnarray}\label{laplacian f2}
\nonumber m^2\Div (\grad f^2)&=&-m^2\langle \grad f,\grad f\rangle+\langle A, \Hess mf\rangle\\
-m^2\Delta f^2&=&-m^2\vert\grad f\vert^2+\langle A,\Hess mf\rangle
\end{eqnarray}
Subsituting \eqref{laplacian f2} into \eqref{Cheng-Yau formula New}, we get
\begin{equation}\label{new versionof CY}
	-\frac{1}{2}\Delta (\vert A\vert^2+2m^2f^2)=\vert\nabla A\vert^2-m^2\vert \grad f\vert^2 +\frac 1 2 \sum_{i,j=1}^{m}R_{ijij}(\lambda_i-\lambda_j)^2
\end{equation}
Considering Lemma \ref{cheninequfor spclike} with equation \eqref{new versionof CY}, we get
\begin{equation}\label{mainequfor surfface}
-\frac{1}{2}\Delta (\vert A\vert^2+2m^2f^2)\geq  -\frac{3m^2(m-2)}{4(m-1)}\vert\grad f\vert^2+\frac 1 2 \sum_{i,j=1}^{m}R_{ijij}(\lambda_i-\lambda_j)^2
		\end{equation}
		Now we can give the proof.
\subsection{The proof of Theorem \ref{surfacetheo}}
First notice that $m=2$ since $M$ is a surface in de Sitter space $\mathbb{S}_1^3(c)$ and
\begin{equation}\label{intRiem>0}
\frac 12\sum_{i,j=1}^m R_{ijij}(\lambda_i-\lambda_j)^2=R_{1212}(\lambda_1-\lambda_2)^2
\end{equation}

We shall show that $\grad f=0$. For the sake of contradiction assume that $\grad f\neq 0$ on a neigbourhood $U$ of $p\in M$. Then we can choose $\lambda_1=\frac{mf}{2}$ on $U$ due to the biconservativity equation \eqref{Agradffor spacelike}. Using \eqref{trace A}, we can say that $\lambda_2=-\frac{3mf}{2}$.

We have $R_{1212} = c - \lambda_1 \lambda_2$ from the Gauss equation. So, we deduce that
\begin{equation}\label{Q}
	R_{1212}(\lambda_1-\lambda_2)^2=\frac 12 (c+\frac 34 m^2f^2)(4m^2f^2). 
\end{equation}
Because of the fact that $m=2$ and \eqref{Q} is non negative, \eqref{mainequfor surfface} implies that
$$\Delta (\vert A\vert^2+2m^2f^2)\leq 0.$$
Because of compactness of $M$, this with \eqref{mainequfor surfface} gives
$$0=\frac 12 (c+\frac 34 m^2f^2)(4m^2f^2),$$
which is impossible. So, $\grad f=0.$ on any $U$.

Therefore $M$ is CMC. The proof is concluded by the Corollary of Akugatawa \cite{akutagawa1987}.\hfill\qed

\textbf{Open Problem}

 Our findings on compact spacelike biconservative hypersurfaces in de Sitter space  under specific geometric constraints, naturally lead to the following question.
 
 \textit{ Must every compact spacelike biconservative hypersurface with constant scalar curvature in de Sitter space be totally umbilic?}
 
	\textbf{Acknowledgements} 
	
	On behalf of all authors, the corresponding author states that there is no conflict of interest.
	
	The manuscript has no associate data.
	\bibliographystyle{abbrv}
	\bibliography{references.bib}

\begin{thebibliography}{10}

\bibitem{akutagawa1987}
K.~Akutagawa.
\newblock On spacelike hypersurfaces with constant mean curvature in the de
  {S}itter space.
\newblock {\em Mathematische Zeitschrift}, 196:13--19, 1987.

\bibitem{rigidityourpaper}
{\c{S}}.~Andronic and A.~Kayhan.
\newblock Rigidity results for compact biconservative hypersurfaces in space
  forms.
\newblock {\em Journal of Geometry and Physics}, 212:105460, 2025.

\bibitem{B}
R.~Bartnik.
\newblock Existence of maximal surfaces in asymptotically flat spacetimes.
\newblock {\em Communications in mathematical physics}, 94:155--175, 1984.

\bibitem{chenrecentdevelopment}
B.-Y. Chen.
\newblock Recent development in biconservative submanifolds.
\newblock {\em arXiv preprint arXiv:2401.03273}, 2024.

\bibitem{cheng1990rkh}
Q.-m. Cheng.
\newblock Complete space-like hypersurfaces of a de {S}itter space with r= kh.
\newblock {\em Memoirs of the Faculty of Science, Kyushu University. Series A,
  Mathematics}, 44(2):67--77, 1990.

\bibitem{cheng1977hypersurfaces}
S.-Y. Cheng and S.-T. Yau.
\newblock Hypersurfaces with constant scalar curvature.
\newblock {\em Mathematische Annalen}, 225:195--204, 1977.

\bibitem{CB}
Y.~Choquet-Bruhat.
\newblock Maximal submanifolds and submanifolds with constant mean extrinsic
  curvature of a lorentzian manifold.
\newblock {\em Annali della Scuola Normale Superiore di Pisa-Classe di
  Scienze}, 3(3):361--376, 1976.

\bibitem{CFM}
Y.~Choquet-Bruhat, A.~E. Fischer, and J.~E. Marsden.
\newblock Maximal hypersurfaces and positivity of mass.
\newblock In {\em Isolated gravitating systems in general relativity}. 1979.

\bibitem{fu2013bi}
Y.~Fu.
\newblock On bi-conservative surfaces in minkowski 3-space.
\newblock {\em Journal of Geometry and Physics}, 66:71--79, 2013.

\bibitem{fu2015explicit}
Y.~Fu.
\newblock Explicit classification of biconservative surfaces in lorentz 3-space
  forms.
\newblock {\em Annali di Matematica Pura ed Applicata (1923-)},
  194(3):805--822, 2015.

\bibitem{fu2016complete}
Y.~Fu and N.~C. Turgay.
\newblock Complete classification of biconservative hypersurfaces with
  diagonalizable shape operator in the minkowski 4-space.
\newblock {\em International Journal of Mathematics}, 27(05):1650041, 2016.

\bibitem{li1997}
H.~Li.
\newblock Global rigidity theorems of hypersurfaces.
\newblock {\em Arkiv f{\"o}r Matematik}, 35(2):327--351, 1997.

\bibitem{N}
S.~Nishikawa.
\newblock On maximal spacelike hypersurfaces in a lorentzian manifold.
\newblock {\em Nagoya Mathematical Journal}, 95:117--124, 1984.

\bibitem{S}
S.~M. Stumbles.
\newblock Hypersurfaces of constant mean extrinsic curvature.
\newblock {\em Annals of Physics}, 133(1):28--56, 1981.

\bibitem{zheng1995}
Y.~Zheng.
\newblock On space-like hypersurfaces in the de {S}itter space.
\newblock {\em Annals of Global Analysis and Geometry}, 13:317--321, 1995.

\end{thebibliography}
\end{document}